\documentclass[a4paper,11pt]{amsart}
\usepackage{indentfirst, amsfonts, amsmath, amsthm, amssymb, amscd}
\usepackage{amsmath,amsfonts,amscd,bezier}
\usepackage{hyperref}
\usepackage{color}
\usepackage{graphicx}
\usepackage{psfrag}
\usepackage[active]{srcltx}
\usepackage{mathrsfs}
\usepackage{tikz}
\usepackage{float}

\newtheorem{theorem}{Theorem}[section]

\newtheorem{question}[theorem]{Question}
\newtheorem{corollary}[theorem]{Corollary}
\newtheorem{proposition}[theorem]{Proposition}
\newtheorem{lemma}[theorem]{Lemma}

\newtheorem{claim}[theorem]{Claim}
\newtheorem{definition}[theorem]{Definition}

\theoremstyle{remark}
\newtheorem{remark}[theorem]{Remark}

\numberwithin{equation}{section}
\DeclareMathOperator{\supp}{supp}

\DeclareMathOperator{\erg}{erg}

\DeclareMathOperator{\phc}{PHC}
\DeclareMathOperator{\mme}{MME}

\title[Characterization of Ergodic Measures of Maximal Entropy  ]{Characterization of Ergodic Measures of Maximal Entropy for Topologically Transitive Partially Hyperbolic Diffeomorphisms with compact center leaves}

\author{Jorge Crisostomo}
\address{Facultad de Ciencias Matem\'aticas, UNMSM, Lima, Per\'u}
\curraddr{}
\email{jcrisostomop@unmsm.edu.pe}
\thanks{This research was supported by Universidad Nacional Mayor de San Marcos RR 004305-R-24 and Project number B24142381-PCONFIGI 2024}
\author{Richard Cubas }
\address{Universidad Cient\'ifica del Sur, Lima, Per\'u}
\curraddr{}
\email{rcubas@cientifica.edu.pe}
\thanks{}

\date{}                                         
\begin{document}
\maketitle

\begin{abstract}
In this paper, we provide an upper bound on the number of maximal entropy ergodic measures with zero Lyapunov exponent for topologically
transitive partially hyperbolic diffeomorphisms with compact one-dimensional center leaves on $\mathbb{T}^3$. Furthermore, we establish a comprehensive characterization of the support of these measures.
\end{abstract}

\section{Introduction}
Topological entropy and metric entropy are fundamental concepts for understanding the intrinsic complexity of a dynamical system. While topological entropy quantifies the complexity of the system by considering all orbits, metric entropy measures the complexity with respect to those orbits that are relevant for a given probability-invariant measure. For diffeomorphisms of a compact Riemannian manifold, the variational principle asserts that the supremum of the metric entropies, taken over all invariant probability measures, coincides with the topological entropy of the system.

A maximal entropy measure is an invariant measure whose metric entropy equals the topological entropy of the system. In this setting, the complexity of the orbits relevant to a maximal entropy measure coincides with the complexity of all orbits of the entire system. A longstanding and fundamental problem is to determine the existence and uniqueness (or finiteness) of ergodic maximal entropy measures, as well as to describe their supports. Sinai, Margulis, Ruelle, and Bowen pioneers in this field established results on the existence and uniqueness of maximal entropy measures for uniformly hyperbolic systems. Newhouse \cite{newhouse1989continuity} proved that any $C^\infty$ diffeomorphism of a compact manifold admits a maximal entropy measure. In the setting of partially hyperbolic dynamical systems, many authors have contributed significant advances to this topic; see \cite{hertz2012maximizing}, \cite{Vaughn2020}, \cite{buzzi2012maximal}, \cite{ures2021maximal}, \cite{ures2012intrinsic}  among others.

\begin{definition}
   A diffeomorphism $f: M \rightarrow M$ is partially hyperbolic, if the tangent bundle splits into three $Df$-invariant subbundles, $T M=E^s \oplus E^c \oplus E^u$, such that for all unit vectors $v^\sigma \in E_x^\sigma,$ $\sigma \in\{s, c, u\}$ with $x \in M$ satisfy
$$
\left\|\left.D f\right|_{E^s}\right\|<1,\left\|\left.D f^{-1}\right|_{E^u}\right\|<1, \text { and }\left\|D f_x\left(v^s\right)\right\|<\left\|Df_x\left(v^c\right)\right\|<\left\|D f_x\left(v^u\right)\right\|
$$
for some appropriate Riemannian metric.
\end{definition}
For partially hyperbolic diffeomorphisms, it is a well-known fact that there are
foliations  $\mathcal{F}^*$ tangent to the distributions $E^*$
for $* = s, u$.  In general it is not true that there is a foliation tangent to $E^c$
. It can fail to be
true even if $dim E^c=1$ for partially hyperbolic dynamics defined on $\mathbb{T}^3$. A partially hyperbolic diffeomorphism $f$ is \textit{dynamically coherent} if there exist  invariant foliations $\mathcal{F}^{c*}$ tangent to $E^{c*}=E^c \oplus E^*$ for $*= s, u$ (and then, there exist invariante foliations $\mathcal{F}^c$ tangent to $E^c$). The leaf of $\mathcal{F}^*$
containing $x$ will be called $\mathcal{F}^*(x)$, for $*= s, u, c, cs, cu$.

We shall say that a set $X$ is $*$-saturated if it is a union of leaves of the strong
foliations $\mathcal{F}^*$ for $*=s$ or $u$. We also say that $X$ is $su$-saturated if it is both $s$- and $u$-saturated. The \emph{accessibility class} of \( x \in M \), denoted \( AC(x) \) is the minimal $su$-saturated set containing $x$. Note that the accessibility classes form a partition of $M$. In case there is some $x\in M$ whose accessibility class is $M$, then the diffeomorphism
$f$ is said to have the accessibility property. This is equivalent to say that any two points  $x,y\in M$  there exists an \( s,u \)-path connecting \( x \) to \( y \); that is, a concatenation of finitely many path segments, each lying entirely within a leaf of either the foliation \( \mathcal{F}^s \) or the foliation \( \mathcal{F}^u \).

In the context of dynamically coherent partially hyperbolic diffeomorphisms defined on a 3-manifold whose center foliation forms a circle bundle, it was proved in \cite{hertz2012maximizing} that if the diffeomorphism is accessible, the number of ergodic measures of maximal entropy is finite. Moreover, the authors established the following dichotomy: either the system has a unique non-hyperbolic maximal entropy measure, or it has a finite number of ergodic measures of maximal entropy, all of which are hyperbolic. More precisely, they showed the following theorem:

\begin{theorem}\label{dichotomy} \cite{hertz2012maximizing}
Let $f : M \rightarrow M$ be a $C^{1+\alpha}$ partially hyperbolic diffeomorphism of a 3-dimensional closed manifold $M$. Assume that $f$ is dynamically coherent with one dimensional compact central leaves and has the accessibility property. Then, $f$ has finitely many ergodic measures of maximal entropy. There are two possibilities:
\begin{enumerate}
\item $f$ has a unique measure of maximal entropy  $\mu$. The central Lyapunov exponent $\lambda_c(\mu)$ vanishes and $(f, \mu)$ is isomorphic to a Bernoulli shift,

\item  $f$ has more than one ergodic  measure of maximal entropy  measure, all of them with non-vanishing central Lyapunov exponent. 
\end{enumerate}
\end{theorem}

More recently, in \cite{ures2021maximal}, the authors fully characterized the number of ergodic measures of maximal entropy for partially hyperbolic diffeomorphisms on 3-manifolds other than $\mathbb{T}^3$. Specifically, they proved the following result:

\begin{theorem} \label{dichotomy1} \cite{ures2021maximal}
 Let $f$ be a $C^2$ partially hyperbolic diffeomorphism on a 3-dimensional nilmanifold $M \neq \mathbb{T}^3$. Then
\begin{itemize}
    \item[(a)]  either $f$ has a unique maximal measure, in which case $f$ is conjugate to a rotation extension of an Anosov diffeomorphism, and the maximal measure is supported on the whole $M$ and it has vanishing center exponent;
\item[(b)] or $f$ has exactly two ergodic maximal measures $\mu^{+}, \mu^{-}$, with positive and negative center Lyapunov exponents, respectively.
\end{itemize}
\end{theorem}
We emphasize that partially hyperbolic diffeomorphisms defined on 3-manifolds, other than $\mathbb{T}^3$, are accessible. In the non-accessible context, the authors of \cite{rocha2022number} addressed the problem of the number of maximal entropy measures for partially hyperbolic diffeomorphisms with a compact central foliation on $\mathbb{T}^3$, under certain hypotheses on the dynamics of the center foliation. More specifically, they proved the following theorem:

\begin{theorem}\cite{rocha2022number}
Let $f: \mathbb{T}^3 \rightarrow \mathbb{T}^3$ be a $C^2$ dynamically coherent partially hyperbolic diffeomorphism with compact center leaves. Suppose that $f$ has a periodic center leaf with 1-Morse-Smale dynamics, then at least one of the following occurs:
\begin{enumerate}
    \item[(1)] Either the number of ergodic maximal measures is precisely two (one with positive center exponent and one with negative center exponent); 
    \item[(2)] or there exists an invariant su-torus.
\end{enumerate}
\end{theorem}
In the previous theorem, by an $su$-torus, we mean a $2$-torus that is an accessibility class. Another result proved in \cite{rocha2022number} was an interesting dichotomy regarding the existence of a periodic leaf with dynamics of irrational rotation number:

\begin{theorem}\cite{rocha2022number}
    Let $f: \mathbb{T}^3 \rightarrow \mathbb{T}^3$ be a $C^2$-partially hyperbolic diffeomorphism, dynamically coherent with compact central leaves with one periodic leaf with irrational rotation number then:
    \begin{enumerate}
        \item[(1)] Either there is unique measure of maximal entropy $\mu$, moreover $\lambda^c(\mu)=0$ and $f$ is conjugate to rotation extension of Anosov homeomorphism or
        \item[(2)] $f$ admits exactly two ergodic measures of maximal entropy. Both measures are hyperbolic and have opposite sign of center Lyapunov exponent.
    \end{enumerate}
\end{theorem}
In the present paper, we address the problem of the number of ergodic measures of maximal entropy with zero central Lyapunov exponent for partially hyperbolic diffeomorphisms with compact central foliation on $\mathbb{T}^3$,  and we provide a characterization of their supports. In Theorem~\ref{Teo1}, assuming transitivity, we prove that such systems admit at most two ergodic measures of maximal entropy with zero central Lyapunov exponent. Moreover, we show a sharp dichotomy: the support of these measures is either the entire $\mathbb{T}^3$ or the orbit of a compact periodic $su$-torus.  In Theorem \ref{teo2}, we show that if the system has a transversely hyperbolic $su$-torus, then the number of ergodic measures of maximal entropy can be either two or three.

\section{Statement of results}
To simplify notation, we will denote by $\phc^{2}(\mathbb{T}^{3})$ the set of systems $f: \mathbb{T}^3 \rightarrow \mathbb{T}^3$ that are $C^2$ dynamically coherent partially hyperbolic diffeomorphisms with compact center leaves and admit global holonomies,  this means that for any $y \in \mathcal{F}^{\sigma}(x)$, the holonomy map $H_{x, y}^{\sigma}: \mathcal{F}^c(x) \rightarrow \mathcal{F}^c(y)$ is a homeomorphism. Here, for $\sigma \in \{s,u\}$, we define $H_{x, y}^{\sigma}(z)=\mathcal{F}^{\sigma}(z) \cap \mathcal{F}^c(y)$.

We address the problem of the number of ergodic measures of maximal entropy and their supports for systems $f\in \phc^{2}(\mathbb{T}^{3})$. For partially hyperbolic systems on any three-dimensional nilmanifold other than  $\mathbb{T}^3$, this problem has been completely resolved by the work of Marcelo Viana and Ures~\cite{ures2021maximal}. In contrast, the case $M=\mathbb{T}^3$ presents a more subtle and intriguing scenario. A key ingredient in the proof of the aforementioned result is the accessibility property, which, in the context of this work, is not necessarily satisfied.  In the case that $f$ is topologically transitive we establish an upper bound for the number of ergodic measures of maximal  entropy  with zero central Lyapunov exponent. This set of measures is denoted by $\mme_0(f)$ 

\begin{theorem}\label{Teo1}
 Let $f\in \phc^{2}(\mathbb{T}^{3})$. If $f$ is topologically transitive, then $$\# \mme_{0}(f)\leq 2.$$ Moreover, if $\mu\in \mme_{0}(f)$, then:
\begin{itemize}
    \item[(1)] $\text{supp}(\mu)=\mathbb{T}^3$, or
    \item[(2)] $\text{supp}(\mu)$ is the orbit of a compact periodic $su$-torus.
\end{itemize}

\end{theorem}

Before proceeding to the proof of the above theorem, we present the following key result, which provides a characterization of the support of maximal entropy measures with zero central Lyapunov exponent.
\begin{theorem}\label{teoaux}
 Let $f\in \phc^{2}(\mathbb{T}^{3})$. If  $f$ admits a measure of maximal entropy  $\mu$ with zero central Lyapunov exponent, $\lambda_c(\mu)=0$, then
 \begin{itemize}
   \item[(1)] $\text{supp}(\mu)=\mathbb{T}^{3}$, or
    \item[(2)] the support of $\mu$ contains a  periodic $su$-torus.
 \end{itemize}  
\end{theorem}
It is known that every hyperbolic homeomorphism  on $\mathbb{T}^{2}$ admits a unique measure of maximal entropy. Therefore, if $\mathbb{T}^{2}_{su}$ is a periodic $su$-torus for $f\in \phc^{2}(\mathbb{T}^{3})$, it is easy to see that $f|_{\mathcal{O}(\mathbb{T}^{2}_{su})}$ admits a unique measure of maximal entropy. We say that $\mathbb{T}^{2}_{su}$ is a transversely hyperbolic $su$-torus if $\mathbb{T}^{2}_{su}$ is periodic, and its unique measure of maximal entropy has a central Lyapunov exponent different from zero, that is,  
$$\int_{\mathbb{T}_{su}^{2}} \log \left\|D f \mid E^{c}\right\|d\eta\neq 0,$$
where  $\eta$ is the unique  measure of maximal entropy for $f|_{\mathcal{O}(\mathbb{T}^{2}_{su})}$. We will show in the presence of a transversely hyperbolic $su$-torus, the system admits exactly two or three ergodic measures of maximal entropy.

\begin{theorem}\label{teo2}
Let $f\in \phc^{2}(\mathbb{T}^{3})$. If $f$ has a transversely hyperbolic $su$-torus, then $$2\leq \#\mme_{\erg}(f)\leq 3,$$
where $\mme_{\erg}(f)$ denotes the set of ergodic measures of maximal entropy.
\end{theorem}

For partially hyperbolic systems in $\mathbb{T}^3$, all examples known so far have at most three ergodic maximum entropy measures (\cite{rocha2022number}). This leads us to the following question: 
\begin{question}
If  $f\in \phc^{2}(\mathbb{T}^{3})$  is topologically transitive, then    $\#\mme_{\erg}(f)\leq 3$?
\end{question}

\section{Disintegration of Maximal Entropy Measures along Increasing Partitions}
Let $M$ be a three-dimensional compact manifold, and $f: M \rightarrow M$ be a $C^2$ partially hyperbolic diffeomorphism with compact center leaves. We define the quotient space $M_c$ and quotient dynamics $f_c$ through the natural projection $\pi_c: M \rightarrow M_c:=M / \mathcal{F}^c$, which sends every point in $\mathcal{F}^c(x)$ to $x_c=\pi_c(x)$. It is known that $M_c$ is homeomorphic to $\mathbb{T}^2$, and $f_c$ is an Anosov homeomorphism (See Theorem 3 in \cite{hertz2012maximizing}). Therefore, it is topologically conjugate to an Anosov diffeomorphism on $\mathbb{T}^2$.

We use $\mathcal{W}^i, i \in\left\lbrace s, u\right\rbrace $, to denote the stable and unstable foliations of $f_c$. For any point $a \in M_c$, by $\mathcal{W}^s(a)$ and $\mathcal{W}^u(a)$, we mean the stable and unstable sets of $a$. Since $f$ is conjugate to an Anosov diffeomorphism on $\mathbb{T}^2$, both (topological) foliations $\mathcal{W}^u$ and $\mathcal{W}^s$ are minimal. Observe that
\begin{itemize}
\item $\pi_c\left(\mathcal{F}^{c s}(x)\right)=\mathcal{W}^s(\pi_c(x))$,
\item $\pi_c\left(\mathcal{F}^{c u}(x)\right)=\mathcal{W}^u(\pi_c(x))$.
\end{itemize}

It is well known that Anosov diffeomorphisms on $\mathbb{T}^2$ have a unique measure of maximal entropy. There exists a relationship among measures  of maximal entropy for $f$ and $f_c$; in fact, $h_{\text{top}}(f)=h_{\text{top}}(f_c)$. Furthermore, every measure of maximal entropy $\mu$ for $f$ projects onto the unique measure of maximal entropy $\eta$ for $f_c$. We can still relate these measures through their decomposition along special partitions, which it is defined below:

\begin{definition}
 A measurable partition $\xi$ is called increasing and subordinated to $\mathcal{F}$ if it satisfies:
\begin{itemize}
\item[(a)] $\xi(x) \subseteq \mathcal{F}(x)$ for $\mu$-almost every $x$,
\item[(b)] $f^{-1}(\xi) \geq \xi$ (increasing property),
\item[(c)] $\xi(x)$ contains an open neighborhood of $x$ in $\mathcal{F}(x)$ for $\mu$-almost every $x$.
\end{itemize}
\end{definition}

The existence of an increasing measurable  partition subordinated to an invariant lamination, in general, is a difficult problem. For an uniformly expanding invariant foliation  by a diffeomorphism, there always exists  such a partition (See \cite{yang2021entropy}). We say that an invariant foliation $\mathcal{F}$ is uniformly expanding if there exists $\alpha>1$ such that $\forall x \in M,\left|D f \mid_{T_{x}(\mathcal{F}(x))}\right|>\alpha$. Observe that if $\mathcal{F}$ is an expanding foliation, we have more useful properties for a partition $\xi$ that is increasing and subordinated:
\begin{itemize}
\item[(d)] $\bigvee_{n=0}^{\infty} f^{-n} \xi$ is the partition into points;
\item[(e)] the largest $\sigma$-algebra contained in $\bigcap_{n=0}^{\infty} f^{n}(\xi)$ is $\mathcal{B}_{\mathcal{F}}$ where $\mathcal{B}_{\mathcal{F}}$ is the $\sigma$-algebra of $\mathcal{F}$-saturated measurable subsets (union of entire leaves).
\end{itemize}

From \cite{ures2021maximal}, we can obtain an increasing partition $\xi$ for $f$ along $\mathcal{F}^u$ and  an increasing partition $\xi_{c}$ along $\mathcal{W}^u$ such that $\pi_c\mid_{\xi(x)}:\xi(x)\rightarrow \xi_{c}(x_c)$ is a homeomorphism. 

In fact, let $\eta$ be the unique measure of maximal entropy for $f_c$ on $M_c\equiv \mathbb{T}^2$. Take ${\tilde{B}_1, \cdots, \tilde{B}_k}$ to be a Markov partition for $f_c: M_c \rightarrow M_c$, such that $\eta\left(\cup_{1 \leq i \leq k} \partial \tilde{B}_i\right)=0$. For every $x_c \in \tilde{B}_i \subset M_c$, let $\mathcal{W}_{\text{loc}}^u(x_c)$ be the connected component of $\mathcal{W}^u\left(x_c\right) \cap \tilde{B}_i$ that contains $x_c$, and let $\xi_{c}$ be the partition of $M_c$ whose elements are those local unstable sets. This partition $\xi_{c}$ is measurable and increasing. Let $\{\eta_{x_c}^u: x_c \in M_c\}$ be a  disintegration of $\eta$ relative to this partition (in the sense of Rokhlin \cite{rohlin1949fundamental}). Now, denote $B_i=\pi_c^{-1} (\tilde{B}_i)$ for each $i$. For every $x \in B_i \subset M$, let $\mathcal{F}_{\text {loc }}^u(x)$ be the connected component of $\mathcal{F}^u(x) \cap B_i$ that contains $x$. Denote by $\xi$ the partition of $M$ whose atoms are those local unstable leafs. It is clear that $\xi$ is an increasing partition for $f$. Thus, $\pi_c\mid_{\xi(x)}:\xi(x)\rightarrow \xi_{c}(x_c)$ is a homeomorphism.

We define $m_x^{u}:= (\pi_c)^{*}\eta_{x_c}^u$. For $x$ and $y$ in the same $B_i$, and any $z \in \xi(x)$ we denote by $\mathcal{H}_{x, y}^{c s}(z)$ the unique point in the intersection of $\mathcal{F}_{\text {loc }}^{c s}(z)$ with $\xi(y)$. The map $\mathcal{H}_{x, y}^{c s}: \xi(x) \rightarrow \xi(y)$  is called the center-stable holonomy map from $x$ to $y$.
\begin{lemma}\label{lemaholonomia}
For $\eta$-almost any points $x_c, y_c \in \tilde{B}_i$, we have that $$\left(\mathcal{H}_{x, y}^{cs}\right)_* m_x^u = m_y^u, \text{ for all } x\in \pi_c^{-1}(x_c) \text{ and } y\in \pi_c^{-1}(y_c).$$ 
\end{lemma}

\begin{proof}
Let $x_c, y_c \in \tilde{B}_i$, define the stable holonomy map from $x_c$ to $y_c$, $\tilde{\mathcal{H}}_{x_c, y_c}^s: \xi_{c}\left(x_c\right) \rightarrow \xi_{c}\left(y_c\right)$, where for any $z \in \xi_{c}\left(x_c\right)$,  $\tilde{\mathcal{H}}_{x_c, y_c}^s(z)$ is the unique point where $\mathcal{W}_{\text{loc}}^s(z)$ intersects $\xi_{c}\left(y_c\right)$. Up to conjugacy, we may view $f_c$ as a linear Anosov torus diffeomorphism, and then the stable holonomy maps $\tilde{\mathcal{H}}_{x_c, y_c}^s$ are affine. Therefore, the measure of maximal entropy $\eta$ is the Lebesgue area, and each $\eta_{x_c}^u$ corresponds to the normalized Lebesgue length along the plaque $\xi_{c}\left(x_c\right)$. Since affine maps preserve normalized Lebesgue length, we have that 
\begin{equation}\label{propiedad1}
    \left(\tilde{\mathcal{H}}_{x_c, y_c}^s\right)_* \eta_{x_c}^u = \eta_{y_c}^u, \text{ for }\eta\text{-almost any points } x_c, y_c \in \tilde{B}_i.
\end{equation} 
Let $C_i \subset \tilde{B}_i$ be the set of points that satisfy \ref{propiedad1}, then $\eta(C_i) = \eta(\tilde{B}_i)$. Let $x_c, y_c \in C_i$, if $x \in \pi_c^{-1}(x_c)$ and $y \in \pi_c^{-1}(y_c)$, for any measurable set $I \subset \xi^{u}(y)$, we have that
\begin{eqnarray*}
   \left(\mathcal{H}_{x, y}^{cs}\right)_* m_x^u(I) &=& m_x^u((\mathcal{H}_{x, y}^{cs})^{-1}(I))\\
   &=& \eta_{x_c}^u(\pi_c((\mathcal{H}_{x, y}^{cs})^{-1}(I)))\\
    &=& \eta_{x_c}^u((\mathcal{H}_{x_c, y_c}^s)^{-1}(\pi_c(I)))\\
    &=& (\tilde{\mathcal{H}}_{x_c, y_c}^s)_* \eta_{x_c}^u(\pi_c(I))\\
    &=& \eta_{y_c}^u(\pi_c(I))\\
    &=& m_y^u(I).
\end{eqnarray*}
This concludes the proof of the Lemma.
\end{proof}

\begin{remark}\label{remark1}
The measure $m_x^u := (\pi_c)^{*}\eta_{x_c}^u$ is not defined for all $x \in M$ because  $\eta_{x_c}^u$ is defined for $\eta$-almost every point $x_c \in M_c$. However, by Lemma \ref{lemaholonomia}, we can extend $m^{u}_z$ to every $z \in M$. Indeed, given $z \in M$, there exists $i \in {1,\ldots,k}$ such that $z \in B_i$. Take $x \in \pi_c^{-1}(C_i)$ for $C_i$  defined in the previous lemma and consider the cs-holonomy $\mathcal{H}_{x, z}^{cs}: \xi(x) \rightarrow \xi(z)$. Then define $m^{u}_z = (\mathcal{H}_{x, z}^{cs}){}m^{u}_x$, thus we obtain a family ${m^{u}_x}$ defined for all $x \in M$.
\end{remark}
Using this remark we have the following lemma:
\begin{lemma}\label{lemaprop}
For every $x \in M$,
$$
f_*^{j}\left(\left.m_{f^{-j}(x)}^u\right|_{f^{-j}(\xi(x))}\right)=m_{f^{-j}(x)}^u\left(f^{-j}(\xi(x))\right) \cdot m_{x}^u
$$
\end{lemma}
\begin{proof}
Since $\{\eta_{x_c}^{u}\}$ is the decomposition of $\eta$ with respect to $\xi_{c}$, let $$\tilde{\eta}_{x_c}:= \dfrac{(f_c)_*(\eta_{f^{-j}_{c}(x_c)}^u|_{f^{-j}_{c}(\xi_{c}(x_c))})}{\eta_{f^{-j}_c(x_c)}^u(f^{-j}_c(\xi_{c}(x_c))) }.$$
It is easy to see that $\tilde{\eta}_{x_c}$ is a probability measure, and the family $\{\tilde{\eta}_{x_c}\}$ is a disintegration for $\eta$. By the uniqueness of disintegration, we have that $\tilde{\eta}_{x_c}=\eta^{u}_{x_c}$ for almost every $x_c$. Therefore, 
\begin{equation}\label{eqcu1}
    (f_c)_*(\eta_{f^{-j}_{c}(x_c)}^u|_{f^{-j}_{c}(\xi_{c}(x_c))})=\eta_{f^{-j}_c(x_c)}^u(f^{-j}_c(\xi_{c}(x_c))) \cdot \eta_{x_c}^u
\end{equation}
Since $m_x^u:= (\pi_c)^{*}\eta_{x_c}^u$, this Lemma is a consequence of \ref{eqcu1} and Remark \ref{remark1}.
\end{proof}

\begin{theorem}\cite{tahzibi2019invariance}\label{equiv}
  If $\mu$ is measure of maximal entropy for $f$ such that $\lambda_c(\mu)\leq  0$ then $\mu_x\equiv m^{u}_x$ for $\mu$-almost every point $x\in M$.
\end{theorem}
The previous theorem is a particular case of Theorem A from \cite{tahzibi2019invariance}, where the measures $m^{u}_x$ are referred to as Gibbs measures.
\begin{corollary}\label{saturado}
 Let  $\mu$ be a measure of maximal entropy for $f$:
\begin{itemize}
    \item[(a)]  If $\lambda_c(\mu)\leq  0$ then $\supp(\mu)$ is $u$-saturated set.
    \item[(b)]  If $\lambda_c(\mu)\geq  0$ then $\supp(\mu)$ is $s$-saturated set.
\end{itemize}
   So, if  $\lambda_c(\mu)= 0$ then $\supp(\mu)$ is $s,u$-saturated set.
\end{corollary}
\begin{proof}
It suffices to show that $\xi(x) \subset \supp \mu$ for $\mu$-almost every $x$. By the  Theorem  \ref{equiv}, $\left(\pi_c\right)_*\left(\mu_x^u\right)=\eta_{x_c}^u$ for every $x$. Note that $\eta_{x_c}^u$ is supported on the whole $\xi_{c}\left(x_c\right)$, since it corresponds  to the normalized Lebesgue measure on $\xi_{c}\left(x_c\right)$. Thus $\operatorname{supp} \mu_x^u=\xi(x)$. Moreover, for $\mu$-almost every point $x, \mu_x^u$-almost every point is a regular point of $\mu$, hence $\xi(x)=\operatorname{supp}\left(\mu_x^u\right) \subset \operatorname{supp} \mu$, for $\mu$-almost every point.    
\end{proof}

\begin{proposition}\label{propositmed}
Given $x \in M$, then any weak limit of the sequence of probability measures
$$
\mu_n:=\frac{1}{n} \sum_{j=0}^{n-1}\left(f^j\right)_* m_x^u
$$
is a measure of maximal entropy.   
\end{proposition}
\begin{proof}
By Theorem \ref{equiv}, it is enough to prove that if $\mu$ is a weak limit of $(\mu_n)_{n\in \mathbb{N}}$, then the decomposition of $\mu$ along $\xi$ is given by $\{m^{u}_x\}_{x\in M}$.  Consider $\mu_n$ and $j\in \{0,1,\ldots,n\}$, since $\xi$ is an increasing measure, we have that if $f^j(\xi(x))\cap B_i \neq \emptyset$, then $G_i=f^j(\xi(x))\cap B_i$ is the union of (complete) elements of $\xi$ inside $B_i$. Then  $G_j=\sqcup^{n_j}_{\ell=1}\xi(z_\ell^{j})$ for some $z_1^{j},\ldots , z_\ell^{j}\in B_i$. Using the Lemma \ref{lemaprop} we obtain
\begin{eqnarray*}
    (\left(f^j\right)_* m_x^u)|_{B_i}=  (\left(f^j\right)_* m_x^u)|_{G_j}&=&\sum_{\ell=1}^{n_j} (\left(f^j\right)_* m_x^u)|_{\xi(x_{\ell})}\\
    &=&\sum_{\ell=1}^{n_j} f^j_* \big(m_x^u|_{f^{-j}(\xi(x_{\ell}))}\big)\\
     &=& \sum_{\ell=1}^{n_j} m_{x}^u\left(f^{-j}(\xi(x_{\ell}))\right) \cdot m_{x_\ell}^u
\end{eqnarray*}
Now we analyze the disintegration of $\frac{1}{\mu_n(B_i)} \mu_n$ (normalization of $\mu_n$ on $B_i$ ) along the plaque partition inside $B_i$. We identify the quotient measure with respect to the plaques partition:
\begin{eqnarray*}
     \frac{1}{n \mu_n(B_i)} \sum_j  (\left(f^j\right)_* m_x^u)|_{B_i}&=&\frac{1}{n \mu_n(B_i)}\left(\sum_j \sum_{\ell=1}^{n_j} m_{x}^u\left(f^{-j}(\xi(x_{\ell}))\right) \cdot m_{x_\ell}^u \right) \\
\end{eqnarray*}
So, it is enough to consider the quotient measure on the space of plaques by
$$
\sum_j \sum_\ell\dfrac{ m_{x}^u\left(f^{-j}(\xi(x_{\ell}))\right)\delta_{j,\ell} }{n \mu_n(B_i)}
$$
where $\delta_{j, \ell}$ represents the Dirac measure on the element of the quotient space corresponding to $\xi(z^{j}_{\ell})$. Therefore, the disintegration of $\frac{1}{\mu_n(B_i)} \mu_n$ along $\left.\xi\right|_{B_i}$ is given by $\{m^{u}_{z}\}_{z\in B_i}$. By Lemma \ref{lemaholonomia}, the measures $\left\lbrace m^{u}_{z}\right\rbrace _{z\in B_i}$ depends continuously, then we have that every accumulation point for this sequence is a sum of measures $m^{u}_x$ on each $B_i$. Without loss of generality, suppose that $\mu_n \rightarrow \mu$. If $\varphi\in C^0(B_i)$, then
$$
\begin{aligned}
\int_{B_i} \varphi d \mu=\lim _{n \rightarrow \infty} \int_{B_i}\varphi d \mu_n & =\lim _{n \rightarrow \infty} \int_{B_i} \int_{\xi(z)}  \varphi d m_z^u d \mu_n \\
& =\int_{B_i} \int_{\xi(z)}  \varphi d m_z^u d \mu
\end{aligned}
$$

The last equality comes from the continuity of $m^{u}_x$ given by Lemma \ref{lemaholonomia}. The uniqueness of disintegration implies that conditional measures of $\mu$ are given by $\{m^{u}_x\}_{x\in M}$. This completes the proof of the proposition.
\end{proof}

\begin{corollary}\label{existencia}
    If $\Lambda$ is a minimal $u$-saturated set then there is a measure of maximal entropy $\mu$ with $\supp(\mu)=\Lambda$.
\end{corollary}
\begin{proof}
  Let $x\in  \Lambda$, and $\mu$ be a limit of the sequence of probability measures $$\mu_n\doteq\frac{1}{n} \sum_{j=0}^{n-1}\left(f^j\right)_* m_x^u$$
  as stated in Proposition \ref{propositmed}. By  Corollary \ref{saturado} we have that  $\supp(\mu)$ is $u$-saturated set. Moreover, $\supp(\mu)\subset \Lambda$. Indeed, given $y\in M\backslash \Lambda$ then there exists $\epsilon>0$ such that $B(y,\epsilon) \cap \Lambda= \emptyset$. We know that $f^{n}(\xi(x))\subset \Lambda$ for every $n\in \mathbb{N}$, we have that  $f^{n}(\xi(x))\cap B(y,\epsilon)=\emptyset$ for all $n\in \mathbb{N}$. This implies that $\mu_n(B(y,\epsilon))=0$ for all $n\in \mathbb{N}$, thus $\mu(B(y,\epsilon))=0$. So, $y\notin \supp(\mu)$ for every $y\in M\backslash \Lambda$, i.e. $\supp(\mu)\subset \Lambda$. Finally, as $\supp(\mu)$ is closed, $u$-saturated  and $f$-invariant set, we concluded that $\supp(\mu)=\Lambda$.
\end{proof}

 \begin{proposition} Let $\mu$, $\nu$ be an ergodic measure of $u$-maximal entropy with  $\lambda_c(\mu)<0$ and   $\lambda_c(\nu)\leq 0$. Then $$\mu(\supp(\mu) \cap \supp(\nu))=0$$
\end{proposition}
\begin{proof}
 Let $K=\supp(\mu) \cap \supp(\nu)$. Suposse that $\mu(K)>0$ then $K\neq \emptyset$. Since   $\supp(\mu)$ and   $\supp(\nu)$ are invariant compact sets we have that $K$ is invariant  compact set. Since $\mu(K)>0$ and $K$ is invariant  compact set,  the ergodicity of $\mu$ implies that 
 $$K=\supp(\mu).$$  
 
 As $\mu$ and $\nu$ are ergodic measures of maximal $u$-entropy  the  Theorem \ref{equiv} implies that their conditional measures along the increasing partition $\xi$ for $\mu$ and $\nu$ are both given by  the system $\left\{m_{x}^{u}\right\}_{x \in M}$. Let 
 \begin{eqnarray*}
     \mathcal{B}(\mu)&:=&\{x \in M:\frac{1}{n} \sum_{k=0}^{n-1} \delta_{f^{k} x} \rightarrow \mu\}\; \text{and}\\
     \mathcal{B}(\nu)&:=&\{x \in M:\frac{1}{n} \sum_{k=0}^{n-1} \delta_{f^{k} x} \rightarrow \nu\}
 \end{eqnarray*}
 be the ergodic basin of $\mu$ and $\nu$, respectively. Since $\mu$ and $\nu$ are  ergodic measures  we have that $\mu\left(  \mathcal{B}(\mu))\right)=1$ and $\nu\left(  \mathcal{B}(\nu))\right)=1$. Therefore    $m_{x}^{u}\left(K \backslash  \mathcal{B}(\mu)\right)=0$ for $\mu-$ a.e $x\in K$ and  $m_{y}^{u}\left(\supp(\nu)\backslash  \mathcal{B}(\nu)\right)=0$ for $\nu$- a.e $x\in \supp(\nu)$. \\
 
 As the center Lyapunov exponent of $\mu$  is negative, for $\mu$ -a.e. $x\in  \mathcal{B}(\mu)$ and  $ m_{x}^{u}(K\backslash\left. \mathcal{B}(\mu)\right)=0$ for $\mu$-a.e. $x \in  \mathcal{B}(\mu)$   there is $x_0\in K\cap  \mathcal{B}(\mu)$ and  $L \subset\xi(x_0) \cap  \mathcal{B}(\mu)$ with $m_{x}^{u}(L)>0$ and large size of the Pesin local stable manifolds: $\inf_{z\in L}\text{diam}^{cs}(\mathcal{W}_{\text {loc }}^{s}(z))>\delta$,  for some $\delta>0$. Let $i\in {1,\ldots,k}$ such that $x_0\in B_i$ for $B_i$ defined at the beginning of the section.  As $L \subset \xi(x_0)$, there exists  $\epsilon>0$ such that for all $y\in  \mathcal{B}(x_0,\epsilon)\subset B_i$ there is a $(cs,\delta)$-holonomy $\mathcal{H}_{x_0, y}^{c s}:L \rightarrow \mathcal{H}_{x_0, y}^{c s}(L) \subset \xi(y)$.\\

As  $K\subset\supp(\nu)$, there exists  $y_0 \in  \mathcal{B}(\nu)\cap B(x_0,\epsilon)$ with $m_{y_0}^{u}\left(\supp(\nu)\backslash B(\nu)\right)=0$. Since $y_0 \in B(x_0,\epsilon)$, then there is a local $cs$-holonomy $\mathcal{H}_{x_0, y_0}^{c s}: L \rightarrow \xi(y_0)$. This holonomy is absolutely continuous from $\left(L, m_{x}^{u}\right)$ to $\left(\xi(y_0), m_{y_0}^{u}\right)$, hence $m_{y}^{u}(\mathcal{H}_{x_0, y_0}^{c s}(L))>0$.
 
\begin{figure}[H]
\centering
 \includegraphics[scale=0.23]{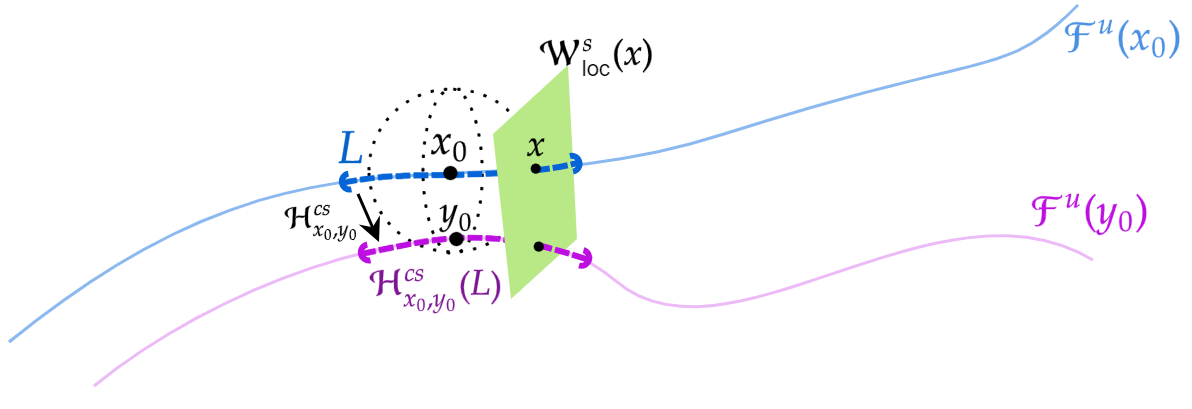}
\caption{$cs$-holonomy $\mathcal{H}_{x_0, y_0}^{c s}$}
\end{figure}
By the choice of $L, \mathcal{H}_{x_0, y_0}^{c s}(x)$ belongs to the Pesin stable manifold of $x$. Since the ergodic basin is saturated by stable manifolds, $\mathcal{H}_{x_0, y_0}^{c s}(L) \subset  \mathcal{B}(\mu)$ and therefore $m_{y_0}^{u}\left( \mathcal{B}(\mu)\right)>0$. As $m_{y_0}^{u}\left(\supp(\nu) \backslash  \mathcal{B}(\nu))\right)=0$, we conclude that $ \mathcal{B}(\nu) \cap  \mathcal{B}(\mu)\neq \emptyset$ and consequently $\mu=\nu$. This  complete the proof of the proposition.
\end{proof}
\subsection{Twin Measures}

In this section, we present a theorem, whose proof  can be found in \cite{Ali}. For every hyperbolic measure, another measure called a twin measure can be constructed, but with the opposite sing in the central Lyaponuv exponent . This allows us to obtain new measures hyperbolic of maximal entropy, that initially appeared in \cite{hertz2012maximizing}. A similar argument can also be found in \cite{diaz2019structure}.

 Let us assume that center foliation is one-dimensional with orientable compact leaves and suppose that $f$ preserves this orientation.

\begin{proposition} [\cite{hertz2012maximizing},  \cite{Ali}] \label{Twin measure} 
If $\nu$ is an ergodic measure of maximal entropy for $f \in \phc^2(M)$ with $\lambda_c(\nu) < 0$, then there is another $f$-invariant probability measure $\nu^*$ which is isomorphic to $\nu$ and with exponent $\lambda_c\left(\nu^*\right) \geq 0$. Moreover, there is a measurable set $Z \subset M$ with $\nu(Z)=1$ such that the restriction the next is map is an isomorphism:
$$
\beta: Z \rightarrow \beta(Z), x \mapsto \sup \mathcal{W}_s^c(x, f) \text { and } \nu^*=\beta_* \nu,
$$
where $\mathcal{W}_s^c(x, f):=\left\{y \in \mathcal{F}^c(x): \lim \sup _{n \rightarrow \infty} \frac{1}{n} \log d\left(f^n x, f^n y\right)<0\right\}$ and $\sup \mathcal{W}^c(x)$ is the extreme point of $\mathcal{W}^c(x)$ in the positive direction.
\end{proposition}
\begin{figure}[H]
\centering
 \includegraphics[scale=0.6]{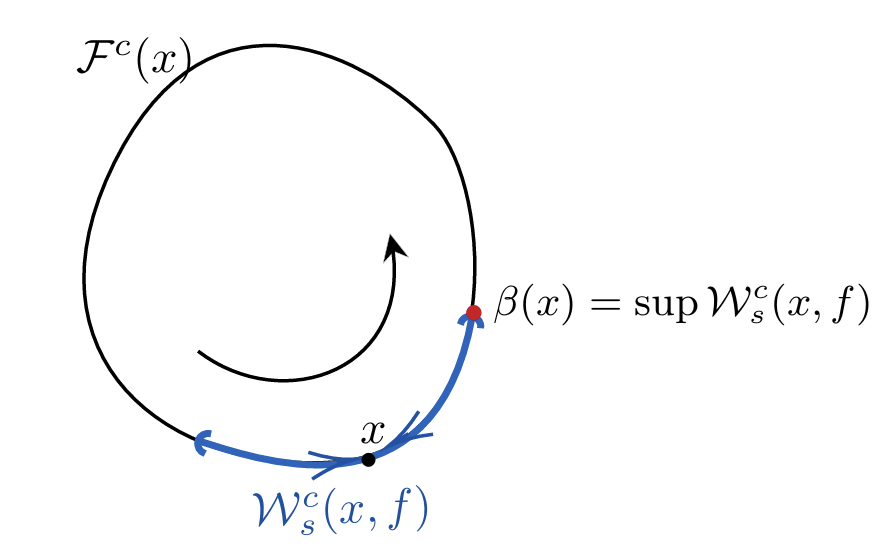}
\caption{Mapping correspondence $\beta$}
\end{figure}
The previous proposition uses in an essential way the one-dimensionality of center leaves. In the setting of Theorem \ref{dichotomy}, it was proved that the twin measure $\nu^*$ is also hyperbolic, i.e. $\lambda_c(\nu^*) > 0.$

\subsection{AB and AI-systems}
In order to prove the main results, we utilize certain properties of non-accessible partially hyperbolic systems, which belong to a class of Infra-AB systems that we will introduce in this section.

\subsubsection{AB-systems}
Suppose $A$ and $B$ are automorphisms of a compact nilmanifold $N$ such that $A$ is hyperbolic and $A B=B A$. Then, $A$ and $B$ define a diffeomorphism
$$
f_{A B}: M_{B} \rightarrow M_{B}, \quad(v, t) \mapsto(A v, t)
$$
on the manifold $M_{B}=N \times \mathbb{R} /(v, t) \sim(B v, t-1)$. Call $f_{A B}$ an $A B$-prototype. \\

 Two partially hyperbolic diffeomorphisms $f$ and $g$ are leaf conjugate if there are invariant foliations $W_{f}^{c}$ and $W_{g}^{c}$ tangent to $E_{f}^{c}$ and $E_{g}^{c}$ and a homeomorphism $h$ such that for every leaf  $L$ in $W_{f}^{c}, h(L)$ is a leaf of $W_{g}^{c}$ and $h(f(L))=g(h(L))$.
 
\begin{definition}
A partially hyperbolic system $f: M \rightarrow M$ is an $A B$-system if it preserves an orientation of the center bundle $E^{c}$ and is leaf conjugate to an AB-prototype.
\end{definition}\label{torus for AB}
A partially hyperbolic system has global product structure if it is dynamically coherent and, after lifting the foliations to the universal cover $\tilde{M}$, the following hold for all $x, y \in \tilde{M}$ :
\begin{itemize}
    \item[(1)] $W^u(x)$ and $W^{c s}(y)$ intersect exactly once,
 \item[(2)] $W^s(x)$ and $W^{c u}(y)$ intersect exactly once,
 \item[(3)] if $x \in W^{c s}(y)$, then $W^c(x)$ and $W^s(y)$ intersect exactly once, and
 \item[(4)] if $x \in W^{c u}(y)$, then $W^c(x)$ and $W^u(y)$ intersect exactly once.
\end{itemize}
\begin{theorem}[\cite{hammerlindl2017ergodic}]
 Every AB-system has global product structure.
\end{theorem}

 For a partially hyperbolic diffeomorphism with one-dimensional center, a $su$-leaf is a complete $C^{1}$ submanifold tangent to $E^{u} \oplus E^{s}$.

\begin{theorem}[\cite{hammerlindl2017ergodic}]\label{hamm}
Every non-accessible $A B$-system has a compact su-leaf.
\end{theorem}
By Proposition 4.2 and Theorem 4.3 from \cite{hammerlindl2017ergodic} we have that:
\begin{theorem}\label{Hamfortorus}
  Let $M$ be a  three dimensional closed manifold and $f:M \rightarrow M$ be a  $C^{2}$ dynamically coherent partially hyperbolic diffeomorphism with compact center leaves. If $f$ preserves orientation of $E^{c}$ and it is not accessible then $f$ is an $AB$-system.
\end{theorem}
In order to consider  systems over infranilmanifolds which do not preserve an orientation of $E^{c}$, we also consider the following generalization introduced in \cite{hammerlindl2017ergodic}.
\begin{definition}
 A diffeomorphism $f_{0}$ is an infra-AB-system if an iterate of $f_{0}$ lifts to an AB-system on a finite cover. 
\end{definition}

As for three-dimensional manifolds, the center foliation $\mathcal{F}^{c}$  is uniformly compact then the following result follows from Corollary 4.9 of \cite{hammerlindl2017ergodic}.

\begin{theorem}\label{is-infra-AB}
 Let $M$ be a  three dimensional closed manifold and $f:M\rightarrow M$  a partially hyperbolic diffeomorphism with compact center leaves. If $f$ is not accessible, then $f$ is an infra-AB-system.
\end{theorem}

 \begin{proposition}[\cite{rodriguez2006accessibility}]
 Suppose $f$ is a partially hyperbolic system with one-dimensional center on a (not necessarily compact) manifold $M$. For $x \in M$, the following are equivalent:
 \begin{itemize}
     \item $AC(x)$ is not open.
     \item $AC(x)$ has empty interior.
      \item $AC(x)$ is a complete $C^{1}$-codimension one submanifold. 
 \end{itemize}
 If $f$ is non-accessible, the set of non-open accessibility classes form a lamination.
 \end{proposition}

\subsubsection{AI-systems}\label{subsection AI-systems}
We now consider partially hyperbolic systems on non-compact manifolds. Suppose $M$ is compact and $f: M \rightarrow M$ is partially hyperbolic. Then, any lift of $f$ to a covering space of $M$ is also considered to be partially hyperbolic. Furthermore, any restriction of a partially hyperbolic diffeomorphism to an open invariant subset is still considered to be partially hyperbolic.

Let $A$ be a hyperbolic automorphism of the compact nilmanifold $N$ and $I \subset \mathbb{R}$ an open interval. The AI-prototype is defined as
$$
f_{A I}: N \times I \rightarrow N \times I, \quad(v, t) \rightarrow(A v, t) .
$$
\begin{definition}
A partially hyperbolic diffeomorphism $f:X\rightarrow X$ on a (non-compact) manifold $X$ is an AI-system if it has global product structure, preserves  orientation of its center direction, and it is leaf conjugated to an AI-prototype.
\end{definition}

Now, consider $f:X\rightarrow X$ an AI-system. Then, there exists a prototype AI-system $f_{AI}:N \times I \rightarrow N \times I$ and a homeomorphism $h:X\rightarrow N\times I$ that conjugate the central leaves of $f$ and $f_{AI}$. Let $\widetilde{X}\rightarrow X$ and $\widetilde{N}\rightarrow N$ be universal coverings. Then, $f$ and $h$ are lifted to functions $\tilde{f}: \widetilde{X} \rightarrow \widetilde{X}$ and $\tilde{h}: \widetilde{X} \rightarrow \widetilde{N} \times I$. Each central leaf of $\tilde{f}$ is of the form $\tilde{h}^{-1}(\{v\} \times I)$ for some $v \in \widetilde{N}$. In general, the liftings of $f$ and $h$ are not unique; they can be chosen so that 

$$\tilde{h}\circ \tilde{f}\circ  \tilde{h}^{-1}(\nu \times I)=A v \times I,$$

where $A: \widetilde {N} \rightarrow \widetilde{N}$ is an automorphism of the hyperbolic Lie group $\widetilde{N}$. Since $A$ fixes the identity element of the Lie group, there exists a central leaf mapped to itself by $\widetilde{f}$, which we denote by $L$. As $L$ is homeomorphic to $\mathbb{R}$, we can give an order to the points of $L$ and define open intervals $(a, b) \subset L$ for $a, b \in L$ and the supremum $\sup J$ for subsets $J\subset L$ in the same way that they are defined in $\mathbb{R}$.

\begin{figure}[H]
    \centering
    \includegraphics[scale=0.4]{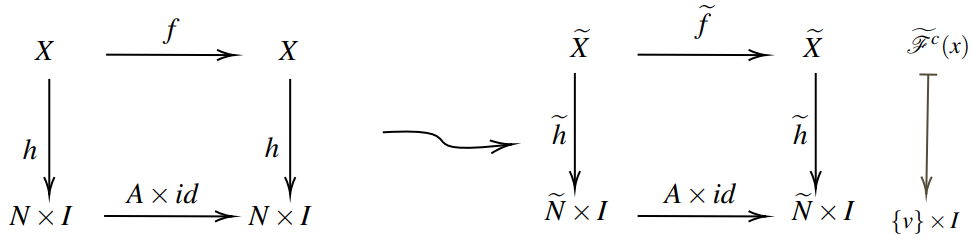}
    \label{fig:my_label}
\end{figure}

 Define a closed subset
$$\Lambda_f=\{t \in L: A C(t) \text{ is not open }\} .$$ 

\begin{lemma}[\cite{hammerlindl2017ergodic}]
Let $f:X\rightarrow X$ be a non-accessible  $AI$-system. Then  $\Lambda_f$ is a non-empty set.
\end{lemma}  
By global product structure, for any $x \in\widetilde{X}$, there is a unique point $R(x) \in L$ such that $W^{u}(x)$ intersects $W^{s}(R(x))$. This defines a retraction, $R:\widetilde{X} \rightarrow L$. By the previous lemma, if $t \in \Lambda_f$, then $R^{-1}(t)=AC(t)$.  Let $\alpha:\widetilde{X} \rightarrow\widetilde{X}$ be a deck transformation of the covering $\widetilde{X} \rightarrow X$. Then,  $\alpha$ defines a map $g_{\alpha} \in \operatorname{Homeo}^{+}(\Lambda_f)$ given by the restriction of $R \circ \alpha$ to $\Lambda_f$ and  $G=\left\{g_{\alpha}: \alpha \in \pi_{1}(\widetilde{X})\right\}.$

\begin{figure}[H]
    \centering
    \includegraphics[scale=0.8]{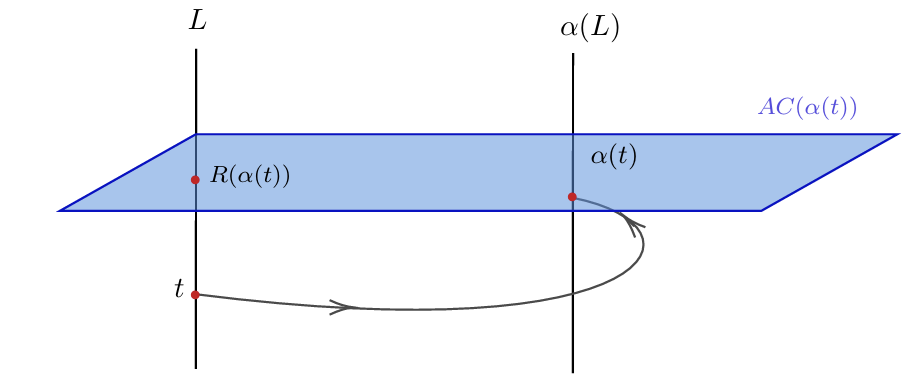}
    \caption{The map $g_{\alpha}$}
\end{figure}

\begin{lemma}\cite{hammerlindl2017ergodic}\label{fix and compact leaf}
 Given a point $t \in \Lambda_f$, $AC(t) \subset\widetilde{X}$ projects to a compact $su$-leaf in $X$ if and only if $t \in \operatorname{Fix}(G)$.
\end{lemma}

\section{Proof of the results}
\subsection{Proof of Theorem \ref{teoaux}}


\begin{proof}[Proof of Theorem \ref{teoaux}]
Let $\mu$ be a measure of maximal entropy  with $\lambda_c(\mu) = 0$ and suppose  that $\supp(\mu)\neq \mathbb{T}^{3}$.  Since  $\lambda_c(\mu) = 0$,  the support  is $s,u$-saturated set (Corollary \ref{saturado}). Also,  $\supp(\mu)\neq \mathbb{T}^{3}$, then we have that    $f$ is not accessible.

Without loss of generality we can assume that $\supp(\mu)$ is a minimal $f$-invariant $u$-saturated  set.  Now, since  $f$ is not accessible by Theorem  \ref{is-infra-AB} we have that $f$ is an infra-AB-system. Therefore,  there exists   an iterate $n_0$ of $f$  that lifts to an AB-system $g$ on a finite cover $\overline{M}$. Consider  $\Lambda$ to be a lifting of $\supp(\mu)$  and $\widehat{\mu}$  an ergodic measure of maximal entropy of $g$ such that $\supp(\widehat{\mu})=\Lambda$ and it is projected on $\mu$. Then  we can suposse that $\Lambda$ is a $g$-invariant minimal $u$-saturated set (Corollary \ref{saturado}) and $\lambda_c(\widehat{\mu})=0$.\\

To prove  theorem \ref{teoaux}, it is enough to prove that   $\Lambda$ is the orbit of a periodic compact leaf $\mathbb{T}_{su}$. Indeed, if    $\Lambda$ is the orbit of a periodic compact leaf $\mathbb{T}_{su}$, then $\pi(\mathbb{T}_{su})\subset \pi(\Lambda)=\supp(\mu)$ is a periodic compact $su$-torus.  
\begin{claim}\label{claimtorus}
$\Lambda$ is a periodic $su$-compact leaf.
\end{claim}
\noindent\textit{Proof of the Claim \ref{claimtorus}.} Since $\widehat{\mu}$ is an ergodic measure of maximal entropy  for $g$ with $\lambda_c(\widehat{\mu})= 0$ then $\widehat{\mu}$ is a measure of maximal entropy for $g^{-1}$. By Corollary \ref{saturado} applied to  $g^{-1} $ we have that $\supp(\widehat{\mu})$ is $s$-saturated set. Moreover, as $\Lambda=\supp(\widehat{\mu})$ is an invariant minimal $u$-saturated set  for $g$ then $\widehat{\mu}$ is a $s,u$-saturated set.  \\

The $AB$-prototype $f_{AB}$ has an invariant center leaf which is a circle. By the leaf conjugacy with a $AB$-prototype, $g$  has an invariant center leaf. Let call this leaf $S$.

Since $\Lambda\cap S\neq S$  is a $g$-invariant set, the Denjoy's theorem implies   that $g|_{S}$ has a rational rotation number $\rho(g|_{S})=p/q\in \mathbb{Q}$.  

Let us comment on the use of the Denjoy's theorem, due to the fact that it needs regularity of the dynamic on the circle. In fact as $g$ is $C^2$ it is enough to show that the center leaf $S$ is also $C^2$. In fact, if $g|_{S}$ has an irrational rotation number then $g|_{S}$ is uniquely ergodic and the Lyapunov exponent of the unique invariant measure is zero. We claim that $S$ is $2-$normally hyperbolic and then center manifold theory implies $C^2$-regularity of $S$.

To prove the claim on $2$-normal hyperbolicity observe that for any $\epsilon > 0$ there is $C > 0$ such that for all $x \in S, n \in \mathbb{N}$ and $v^c \in T_x S:$
$$\|Dg^n_x(v^c)\| \leq C (1+\epsilon)^n.$$

Indeed, if this is not the case, for some $\epsilon > 0$ and any $n$ one would have $x_n$ such that $\|Dg^n_{x_n}(v^c)\| \geq  (1+\epsilon)^n$. Let $\mu$ be any accumulation point of the sequence of empirical measures $\{\frac{1}{n} \sum_{i=0}^{n-1} \delta_{g^i(x_n)}\}$. Then $\mu$ is a $g|_{S}$-invariant measure whose Lyapunov exponent is larger than zero which is a contradiction.

Considering $g^q$ instead of $g$, define 
$$\mathcal{S}(g):=\{x\in S: AC(x) \text{ is a } g \text{ -invariant  compact manifold} \}.$$

Note that $\mathcal{S}(g)$ is a   non-empty closed set, then  $U=S\backslash \mathcal{S}(g)$ is an open subset of  $S$.
 
 \begin{lemma}
 $\Lambda=AC(x)$ for some $x\in \mathcal{S}(g)$.
 \end{lemma}
 \begin{proof}
 Indeed,   if  $\Lambda\neq AC(x)$ for all $x\in \mathcal{S}(g)$ then  $\Lambda\cap U\neq \emptyset$  and there exists a connected component $J$ of $U$ such that  $\Lambda\cap J\neq \emptyset$. Note that $g(J)=J$ because $\rho(g|_{S})=0$.  Since $g(J)=J$  and $\Lambda$ is a minimal compact $u$-saturated and $g$-invariant set,  we have that $\Lambda\subset AC(J)$. Let $J_0=\Lambda\cap J$, define
$$x_{+}=\max J_0.$$
Since $g(J)=J$ and $g(\Lambda)=\Lambda$   we have that $g(J_0)=J_0$ and $x_+\in Fix(g)$. As $AC(x_{+})\subset \Lambda$,
$\Lambda$ is a minimal compact $u$-saturated and $g$-invariant set and $x_+\in Fix(g) \cap \Lambda$, it follows that 
$$\Lambda=cl\big(AC(x_{+})\big).$$
\begin{claim}
$AC(x_{+})$ is not open. Consequently $g|AC(J)$ is non-accessible $AI$-system.
\end{claim}
\begin{proof}
Note that if $AC(x_{+})$ is  open then $AC(x_{+})\cap J$ is open. This implies that there exists $y\in AC(x_{+})\cap J \subset K_{g}$ such that $x_{+}<y$. This implies that $\max J_0>x_0$ which
is a contradiction. Finally,  Lemma 8.9 from \cite{hammerlindl2017ergodic} implies that $g|AC(J)$ is an $AI$-system. As $x_+\in J$ and $AC(x_+)$ is not open, we  conclude  that $g|AC(J)$ is non-accessible.
\end{proof}

Now, as $AC(x_{+})$ is not open and $\Lambda=cl\big(AC(x_{+})\big)$, we have that $J_0\subset \{y\in S: AC(y) \text{ is not open}\}$. From the definition of  $AI$-system, it follows that every center leaf on $AC(J)$ is properly embedded. Therefore, $\Lambda$ intersects each center leaf in a compact set.\\

Let $P:\widetilde{AC(J)} \rightarrow AC(J)$ be the universal cover, the center leaf $L$ of $\widetilde{g}$, $\Lambda_g$ and the group $G$ as in the subsection  \ref{subsection AI-systems}. If $\widetilde{\Lambda}$ and $\widetilde{J}_0$ are the pre-images of $\Lambda$ and $J_0$ by universal cover $P$, respectively, then $\widetilde{\Lambda}$ intersects each center leaf on $\widetilde{AC(J)}$ in a compact set. In particular, $\widetilde{J}_0 = \widetilde{\Lambda}\cap L$ is compact.  Since $J_0\subset \{y\in S: AC(y) \text{ is not open}\}$ we have  that  $\widetilde{J}_{0}\subset \Lambda_{g}$. 
\begin{claim}
If $t_{+}=\max \widetilde{J}_0$, then $t_{+}\in Fix(G)$.
\end{claim}
\begin{proof}
Since $G< Homeo^{+}(\Lambda_g)$, we have that $t_+\leq \beta(t_{+})$ for all $\beta\in G$. Note that $\widetilde{J}_0$ is exactly equal to the orbit $G(\widetilde{J}_0) = \{\beta(t) : \beta \in  G \text{ and } t\in \widetilde{J}_0\}$,i.e. $\widetilde{J}_0=G(\widetilde{J}_0)$. Then $t_+= \beta(t_{+})$ for all $\beta\in G$, this means, $t_+\in Fix(G)$.
\end{proof}
Since $t_{+}=\max \widetilde{J}_0$ and $g(J_0)=J_0$, we have that $t_+\in Fix(\widetilde{g})$ and $P(t_+)=x_+$. Then $t_{+}\in Fix(G)\cap Fix( \widetilde{g})$. By Lemma \ref{fix and compact leaf}, $AC(x_+)\subset AC(J)$ is  an invariant compact manifold, i.e. $x_{+}\in \mathcal{S}(g)$. This contradiction completes the proof of the lemma.
\end{proof}
Since $\Lambda$ is a fixed  compact $su$-torus for $g$ we have that  $\pi(\Lambda)$ is a periodic   $su$- torus for $f$.   This completes the proof of the Theorem.  

\end{proof}

\subsection{Proof of Theorem \ref{Teo1}}
First, we will prove the following Lemma:
\begin{lemma} 
Consider a $C^2$ dynamically coherent partially hyperbolic diffeomorphism $f: \mathbb{T}^3 \rightarrow \mathbb{T}^3$ with compact center leaves. If $f$ is topologically transitive, then  it  have at most two orbits of periodic $su$-torus.
 \end{lemma}
 \begin{proof}
Let $S$ be a periodic center leaf, and let $K_{S}\subset S$ be the set of points $x\in S$ such that $AC(x)$ is a periodic $su$-torus. Assume that $K_{S}$ is non-empty. Since $f$ is topologically transitive, we have $K_S \subsetneq S$. Consequently, $S\backslash K_S$ is a non-empty open set. Let $J$ be a connected component of $S\backslash K_S$, and let $x\in \partial(J)$. Choose $n_0\geq 1$ such that $f^{n_0}(S)=S$. Since $S$ and $AC(\mathcal{O}_{f^{n_0}}(x))$ are transversal compact manifolds, their intersection $S \cap AC(\mathcal{O}_{f^{n_0}}(x))$ is a finite subset of $S$. Hence, $x$ is a periodic point of $f^{n_0}|_S$ for every $x$ on the boundary of $J$. Thus, there exists a natural number $\tau$ that is a multiple of $n_0$ such that $f^\tau(J) = J$.

Then,  $f^{\tau}(AC(J))=AC(f^{\tau}(J))=AC(J)$. If $I$ is another  connected component of $S\backslash K_S$,  since $AC(J)=f^{\tau}(AC(J))$ is an open set  and $f$ is transitive we have that there exists $n\in \{1,\ldots,\tau \}$ such that $f^{n}(AC(J))\cap AC(I)\neq \emptyset$, and  we also have that $AC(I)$ and $AC(J)$ has no $su$-periodic torus we have that $f^{n}(\partial (AC(J))=\partial (AC(I))$. So 
$$K_S= \partial (AC(J)) \cup f(\partial (AC(J)))\ldots \cup f^{\tau-1}(\partial (AC(J))).$$
This completes the proof of the lemma.
 \end{proof}
 
\begin{proof}[Proof of the Theorem \ref{Teo1}]  Suppose that $\#(MME_{0})>1$. By Teorem A in  \cite{hertz2012maximizing}, we have that $f$ is not accessible. Therefore, $f$ is an infra-AB-system and  there exists  $n_0\in \mathbb{N}$ (actually $n_0=1$ or $2$) such that  $f^{n_0}$ lifts to an AB-system  $g$ on a finite cover $\overline{M}$. Let $\tilde{S}$ be a fixed center leaf for $g$, since $\mathbb{T}^{3}$ is not minimal $u$-saturated set  we can suppose that  the rotation number  $\rho(g|_{\tilde{S}})$ is rational, otherwise if $\Lambda$ is $u$-saturated minimal for $x\in \Lambda\cap \tilde{S}$, using Denjoy's theorem, we would have that $\mathbb{T}^{3}=\mathcal{F}^{u}(\tilde{S})=cl(\mathcal{F}^{u}(\mathcal{O}(x)))\subseteq \Lambda \neq \mathbb{T}^{3}$, which is absurd. So, the rotation number  $\rho(g|_{\tilde{S}})$ is rational. Moreover, as $f$ has finitely many periodic $su$-torus  we have that $g$ has finitely many periodic $su$-torus. Up to taking another iterate of $f$ if necessary,  we can assume that all periodic $su$-torus of $g$ are fixed. Let 
  $$K_{\tilde{S}}=\{x \in \tilde{S}: A C(x) \subset \mathbb{T}^{3} \text { is a fixed $su$-torus}\}.$$
 Suppose $J$  is a connected component of  $\tilde{S} \backslash K_{\tilde{S}}$,  by  Lemma 8.9 from \cite{hammerlindl2017ergodic}  this implies that    $\left.f\right|_{A C(J)}$ is an AI-system. Since $f$ is topologically  transitive and $g$  is a  lift of $f^{n_0}$  on a finite cover we have that there exist points $\overline{x}_{1},\ldots, \overline{x}_l\in \overline{M}$ such that 
 $$cl\Big(\mathcal{O}^{+}(\overline{x}_1,g) \cup \ldots \cup \mathcal{O}^{+}(\overline{x}_l,g)\Big)=\overline{M}, $$
 where $\mathcal{O}^{+}(z,g):=\{g^{n}(z):n=0,1,\ldots\}$.  Since $g(AC(J))=AC(J)$ we can suppose that $\{\overline{x}_{1},\ldots, \overline{x}_l\}\cap AC(J)=\{\overline{x}_{1},\ldots, \overline{x}_{\kappa}\}$ for some $\kappa \leq l$. So
 \begin{eqnarray}\label{denso}
 cl\Big(\mathcal{O}^{+}(\overline{x}_1,g) \cup \ldots \cup \mathcal{O}^{+}(\overline{x}_\kappa,g)\Big)=cl\big(AC(J)).
 \end{eqnarray}

 \begin{theorem}[Hammerlindl, Theorem 7.1]\label{teohamm}
  Suppose $g: \hat{M} \rightarrow \hat{M}$ is an AI-system with non invariant compact us-leaves. Then, either
\begin{itemize}
    \item[(1)] $g$ is accessible,
    \item[(2)] there is an open set $V \subset \hat{M}$ such that
$$
\text{cl}(g(V)) \subset V, \quad \bigcup_{k \in \mathbb{Z}} g^{k}(V)=\hat{M}, \quad \bigcap_{k \in \mathbb{Z}} g^{k}(V)=\varnothing,
$$
and the boundary of $V$ is a compact us-leaf, or
\item[(3)] there are non compact us-leaves in $\hat{M}$, uncountable many non-compact us-leaves in $\hat{M}$ and there is $\lambda \neq 1$ such that $g$ is semiconjugate to
$$
\mathbb{T}^2 \times \mathbb{R} \rightarrow \mathbb{T}^2 \times \mathbb{R}, \quad(v, t) \mapsto(A v, \lambda t)
$$
\end{itemize}
 \end{theorem}

\begin{claim}
$g|_{AC(J)}$ is accessible.
\end{claim}
\begin{proof}
Note that  $g|_{AC(J)}$ is an $AI$-system with no invariant compact $su-$ leaves. Then $g|_{AC(J)}$ 
satisfies either case $(1)$,  $(2)$ or case $(3)$  of the theorem above.\\

If  $g|_{AC(J)}$ is non-accessible. Then it  satisfies either case   $(2)$ or case $(3)$. First  suppose that $g|_{AC(J)}$ satisfies the case $(2)$ and let $V\subset AC(J)$ as in  Theorem \ref{teohamm}. Since $\bigcup_{k \in \mathbb{Z}} g^{k}(V)=AC(J)$ and (\ref{denso}),  there is $k_0\in \mathbb{N}$ such that $g^{k_0}(V)\cap \mathcal{O}^{+}(\overline{x}_i,g)\neq \emptyset$ for every $i=1,\ldots, \kappa$.  Therefore $$cl\big(\bigcup_{n\in \mathbb{N}} g^{k}(V)\big)=cl\big(AC(J)\big).$$

Since $cl(g(V)) \subset V$ we have that $g^{n}(V) \subset g(V)$ for every $n\in \mathbb{N}$.
Therefore 
$$cl\big(AC(J)\big)=cl\big(\bigcup_{n\in \mathbb{N}} g^{k}(V)\big)=cl\big(g(V)\big) \subset V \neq AC(J).$$
This contradiction  implies that $g|_{AC(J)}$  does not satisfy the case $(2)$.\\

Now, if $g|_{AC(J)}$ satisfies the case $(3)$. Then, $g|_{AC(J)}$ is  semi-conjugate  to  the map  $A_{\lambda}: \mathbb{T}^{2} \times \mathbb{R} \rightarrow \mathbb{T}^{2} \times \mathbb{R}, \quad(v, t) \mapsto(A v, \lambda t)$ with $\lambda\ne 1$ by  a semi-conjugacy $h:AC(J)\rightarrow \mathbb{T}^{2}\times \mathbb{R}$. Let $z_i=h(\overline{x}_i)$ for $i=1,\ldots,\kappa$, since $h$ is surjective, by (\ref{denso})   we have that
\begin{eqnarray}\label{denso2}
cl\Big(\mathcal{O}^{+}(z_1,A_{\lambda}) \cup \ldots \cup \mathcal{O}^{+}(z_{\kappa},A_{\lambda})\Big)=\mathbb{T}^{2}\times \mathbb{R}.
\end{eqnarray}
On the other hand, since $\mathbb{T}^{2}\times \{0\}$ is invariant  we can assume  that $\{z_1,\ldots,z_{\kappa}\}\subset \mathbb{T}^{2}\times  [(-b,-a)\cup (a,b)]$. Then
\begin{eqnarray*}
\mathcal{O}^{+}(z_1,A_{\lambda}) \cup \ldots \cup \mathcal{O}^{+}(z_\ell,A_{\lambda})&\subset &\mathbb{T}^{2}\times  \big(-b,b\big), \; \text{when } \lambda<1, \text{ and}\\
\mathcal{O}^{+}(z_1,A_{\lambda}) \cup \ldots \cup \mathcal{O}^{+}(z_\ell,A_{\lambda})&\subset& \mathbb{T}^{2}\times  \big((-\infty,-a)\cup (a,\infty)\big), \; \text{when } \lambda>1. 
\end{eqnarray*}
This contradicts  (\ref{denso}). Therefore $g|_{AC(J)}$ does not satisfy the case $(3)$. This completes the proof of the claim.
\end{proof}

Let $T_{su}(f):=\{x\in M: AC(x) \text{ is a periodic torus }\}$.
As $g|AC(J)$ is accessible and $\partial(AC(J))$ projects on  periodic $su$-torus in $\mathbb{T}^{3}$ for all connected component $J$ of $\tilde{S}\backslash K_{\tilde{S}}$  we have that $f$ is accessible  in each connected component of $\mathbb{T}^{3}\backslash T_{su}(f)$.   Since the complement of the union of periodic $su$-torus has finite many connected components, the topological  transitivity implies that there exists $\tau\in \mathbb{N}$ such that 
\begin{eqnarray}\label{clasdeacc}
\quad\mathbb{T}^{3}\backslash T_{su}(f) =AC(x)\cup\ldots \cup f^{\tau}(AC(x))  \;\;\; \text{ for all }\; x\in \mathbb{T}^{3}\backslash T_{su}(f).
\end{eqnarray}

\begin{lemma}
If $f$ has a periodic $su$-torus  and $\mu$ is an ergodic measure of maximal entropy with $\lambda_{c}(\mu)=0$ then $\supp(\mu)$ is the orbit of a periodic $su$-torus.
\end{lemma}
\begin{proof}
Suppose that $\supp(\mu)$ contains a point $x$ such that $AC(x)$ is not a periodic $su$-torus,  then $AC(x)$ is open, by (\ref{clasdeacc}) we have that $cl(AC(x)\cup\ldots\cup f^{\tau}(AC(x)))= \mathbb{T}^{3}$. Then, $\supp(\mu)=\mathbb{T}^{3}$.  We show that this  yields a contradiction.
 Without loss of generality, we assume that  $S=\mathcal{F}^{c}(a)$ is a fixed center  leaf, for some $a\in \mathbb{T}^{3}$.  Let $\{\mu^{c}_{x}\}_{x}$ be the conditionals measures along the center foliation. Then, using the fact that the disintegration is defined for every center leaf (Invariance principle) and it is unique, we have that $\mu^{c}_a$ is an invariant measure of $f|_S$.  Since $f$ has a periodic $su$-torus then $\rho(f|_{S})\in \mathbb{Q}$.  As $\mu$ is fully supported on $\mathbb{T}^{3}$ and the center disintegration is continuous we conclude that $\supp(\mu_{a}^{c}) = S$.  We have that $f|_S$ has rational rotation number and 
$$S=\supp(\mu_{a}^{c})\subset Per(f|_S).$$
Then, $Per(f|_{S})=S$. 

Let $\mathbb{T}_{su} \in T_{su}(f)$ be a periodic torus. If $\mu^c_a$ has an atom in $S \cap \mathbb{T}_{su} $, then by the $su-$invariance of center disintegration we conclude that all the points on $\mathbb{T}_{su}$ are atoms for the center conditional measures and it is easy to see that $\mu$ should be supported on the orbit of such periodic torus which contradicts $\supp(\mu)=\mathbb{T}^{3}.$ 

If this is not the case, then for any small $\epsilon > 0 $ let define $V_{\epsilon}(\mathbb{T}_{su}):= \bigcup_{x\in \mathbb{T}_{su}}B_{c}^{\mu}(x,\epsilon)$
where $B_c^{\mu} (x, \epsilon)$ is a ball around $x$ in the center leaf through $x$ and the conditional measure $\mu^c_x$ of the arc from $x$ to each boundary point of the arc is $\epsilon.$


\begin{figure}[H]
\centering
 \includegraphics[scale=0.23]{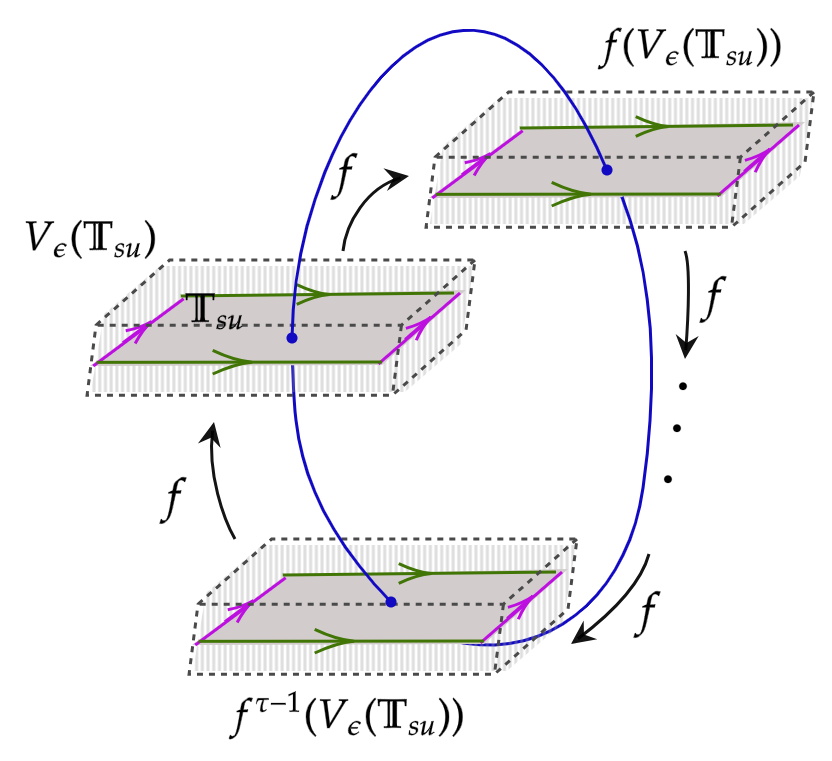}
\caption{Dynamic on $V_{\epsilon}(\mathbb{T}_{su})$ }
\end{figure}
The invariance principle (Theorem D in \cite{avila2015absolute})  implies that $\mu^{c}_{f^{k}(x)}(f^{k}(B_{c}^{\mu}(x,\epsilon))) =B_{c}^{\mu}(f^{k}(x),\epsilon)$. So,  if $\tau$ is the period of $\mathbb{T}_{su}$ we have that 
\begin{eqnarray}\label{epspequeno}
f^{\tau}(V_{\epsilon}(\mathbb{T}_{su}))= V_{\epsilon}(\mathbb{T}_{su}).
\end{eqnarray}
Considering $\epsilon>0$  small,  we have that (\ref{epspequeno})  contradicts the topological transitivity of $f$.
\end{proof}

Finally, since there exist at most two orbits of periodic  $su$- torus and every periodic $su$-torus supports a  unique ergodic measure of maximal entropy, the Lemma implies  that  $\mme_{0}(f)\leq 2$. This completes the proof of Theorem \ref{Teo1}.
\end{proof}

\subsection{Proof of Theorem \ref{teo2} }

\begin{proof}[Proof of Theorem \ref{teo2}]
Let $\mathbb{T}^{2}_{su}$ be the transversely hyperbolic $su$-torus, and $\eta$ be the unique measure of maximal entropy for $f|_{\mathcal{O}(\mathbb{T}^{2}_{su})}$. Without loss of generality, let's assume that $$\int_{\mathbb{T}_{\mathrm{su}}^2} \log \left\|D f \mid E^c\right\| d \eta < 0.$$ Since $f$ is non-accessible, by Theorem \ref{is-infra-AB}, there exists $n_0 \in \mathbb{N}$ such that $f^{n_0}$ lifts to an AB-system $g:\overline{M}\rightarrow \overline{M}$ with a finite covering map $\pi:\overline{M}\rightarrow \mathbb{T}^{3}$. Note also that if $\widehat{\nu}$ is the maximal entropy measure for $g$, then the projection $\nu=\pi_{*} \widehat{\nu}$ is a maximal entropy measure for $f^{n_0}$. Thus, $$\mu=\dfrac{1}{n_0}\sum_{i=0}^{n_0-1} f_{*}^{i}{\nu}$$ is a maximal entropy measure for $f$.

Let $S$ be a fixed central leaf for $g$. The existence of an $su$-periodic torus implies that the rotation number $\rho(g|_S)$ is rational. Without loss of generality, we can assume that every periodic point is fixed, and let $x_1\in Fix(g|S)$ such that $AC(x_1)$ is projected onto $\mathbb{T}_{su}^{2}$. Note that $AC(x_1)$ is also a hyperbolic $su$-torus, denoted by $\overline{\mathbb{T}^{2}}_{su}$. Take $J=(x_1,x_2)\subset S$ such that $AC(x_2)$ is projected onto $\mathbb{T}_{su}^{2}$ and the projection of $AC(z)$ does not intersect $\mathcal{O}(\mathbb{T}_{su})$ for all $z\in (x_1,x_2)$. Therefore, by Lemma 8.9 in \cite{hammerlindl2017ergodic}, $g|AC(J)$ is an $AI$-system.

Let $\widehat{\eta}_1$ be the maximal entropy measure for $g|_{\overline{\mathbb{T}^{2}}_{su}}$. Then, $\lambda_{c}(\widehat{\eta}_1)<0$. Since $\overline{\mathbb{T}^{2}}_{su}$ projects onto $\mathbb{T}^{2}_{su}$. Let's take a maximal entropy measure $\widehat{\nu}$ with $\lambda_{c}(\widehat{\nu})\leq 0$, such that $\supp(\widehat{\nu})$ is the minimal $u$-saturated set most ``close'' to $\overline{\mathbb{T}^{2}}_{su}$. Since $\lambda_{c}(\widehat{\eta}_{1})<0$, there exists $\widehat{\eta}_{2}$, the twin measure of $\widehat{\eta}_{1}$ as in Proposition \ref{Twin measure}, taking the positive orientation of the central leaf. If $\lambda_c(\widehat{\eta}_{2})=0$, then we take $\widehat{\nu}=\widehat{\eta}_{2}$. If $\lambda_c(\widehat{\eta}_{2})>0$, again by Proposition \ref{Twin measure}, taking the positive orientation of the central leaf, there exists a maximal entropy measure $\widehat{\eta}_{3}$ with $\lambda_{c}(\widehat{\eta}_{3})\leq 0$. Take $\widehat{\nu}=\widehat{\eta}_{3}$.

Identifying each central leaf of $AC(J)$ with $(0,1)$, we can consider $AC(J)\equiv \mathbb{T}^{2}\times(0,1)$. Let $\Lambda=\supp(\widehat{\nu})\subset \mathbb{T}^{2}\times(0,1]$. Note that $\Lambda\cap I_z\neq \emptyset$ for every $z\in AC(J)$. For each $(z, t) \in AC(J)=\mathbb{T}^{2} \times (0,1)$, let $t_z^{-}=\min ( I_z\cap \Lambda)$. Define the sets
\begin{eqnarray*}
V=\left\{(z, t) \in AC(J): t<t_z^{-} \right\} \text{ and }
V^{+}=\left\{(z, t) \in AC(J): t\geq t_{z}^{-}\right\}. 
\end{eqnarray*}
\begin{figure}[H]
\centering
 \includegraphics[scale=0.23]{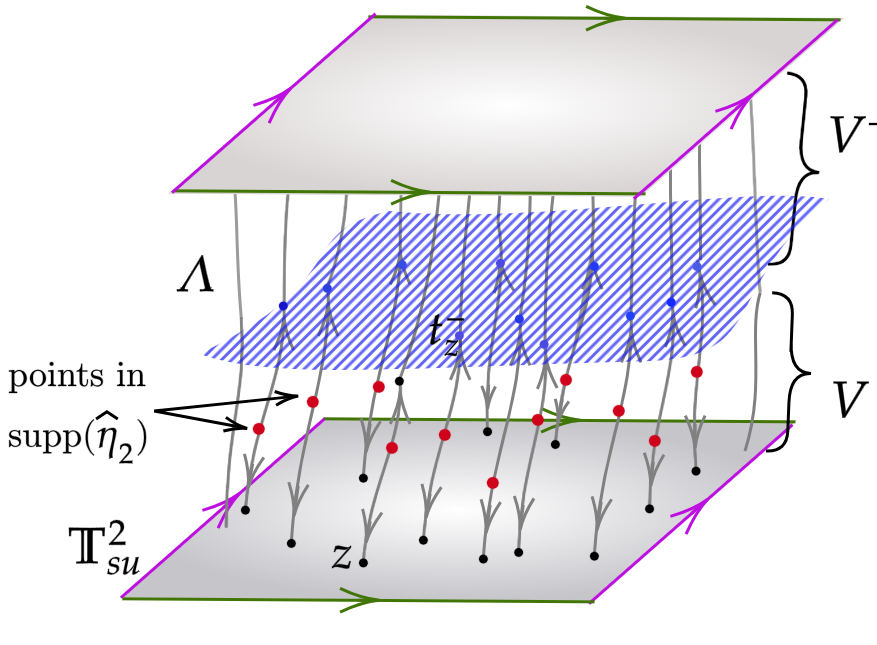}
\caption{support of $\widehat{\eta}_3$}
\end{figure}
Since $g$ preserves the orientation of the central fibers, we have that $V$ and $V^{+}$ are invariant subsets. Moreover, as $\Lambda$ is compact, there exists $\delta>0$ such that $I_{z} \cap V$ is an arc of length greater than $\delta$, i.e., $\delta<t^{-}_z\leq 1$ for every $z\in \mathbb{T}^{2}$.

Let's show that $V^{+}$ is a closed subset of $AC(J)$. Indeed, let $\left((z_ {n}, t_{n})_{n\in \mathbb{N}}\right)\subset  V^{+}$ be a sequence of points such that $\left(z_ {n}, t_{n}\right) \rightarrow(z, a)\in AC(J)$. Since $\delta\leq t^{-}_{z_n}\leq t_{n}$, then $(z,a)\in \mathbb{T}^{2}\times[\delta,1]$. Also, as $\Lambda$ is compact, taking a subsequence, we can assume that $(z_n,t^{-}_{z_n})\rightarrow (z,\omega^{-})\in \Lambda$. Thus, $$t^{-}_z\leq \omega^{-}\leq a.$$ Therefore, $(z,a)\in V^{+}$, i.e., $V^{+}$ is closed in $AC(J)$. Hence, we have that $V$ is an open $g$-invariant set. As $V$ is $g$-invariant, there exists $k_0\in \mathbb{N}$ such that $f^{k_0}(\pi(V))=\pi(V)$. Now, define
\begin{eqnarray*}
U&=&\bigcup_{i=0}^{k_0}f^{i}(\pi(V)).
\end{eqnarray*}
Since $U$ is an open $f$-invariant set, then $\overline{U}=M$. Moreover, note that $\partial(U)\subset (\mathcal{O}(\overline{\mathbb{T}^{2}}_{su})\cup K)$, with $K=\mathcal{O}(\pi(\Lambda))$. By the definition of $V$, if $\widehat{\mu}$ is a maximal entropy measure with $\supp(\widehat{\mu})\cap V\neq \emptyset$, then $\widehat{\mu}=\widehat{\nu}$. As the maximal entropy measures of $g$ are projected surjectively onto the maximal entropy measures of $f^{n_0}$. It is clear that if $\mu'\in \mme(f^{n_0})$ with $\supp(\mu')\cap \pi(V)\neq \emptyset$, then $\mu'=\pi_{*}\widehat{\nu}$. This implies that if $\mu\in \mme(f)$ with $\supp(\mu)\cap U\neq \emptyset$, then
$$\mu=\dfrac{1}{n_0}\sum_{i=0}^{n_0-1} f_{*}^{i}{\nu}, \;\; \pi_{*}(\widehat{\nu})=\nu.$$ If $\eta_1= \pi_{*}(\widehat{\eta}_{1})$ and $\eta_3= \pi_{*}(\widehat{\eta}_{3})$, then the only possible maximal entropy measures for $f$ are $\mu_1=\dfrac{1}{n_0}\sum_{i=0}^{n_0-1} f_{*}^{i}{\eta_{1}}$, $\mu_2=\dfrac{1}{n_0}\sum_{i=0}^{n_0-1} f_{*}^{i}{\nu}$, and $\mu_3=\dfrac{1}{n_0}\sum_{i=0}^{n_0-1} f_{*}^{i}{\eta_{3}}$ (note that $\mu_1=\eta$ and it can also happen that $\mu_1=\mu_3$). This concludes the proof of the theorem.
\end{proof}

\begin{remark}
In \cite{rocha2022number}, examples of $f\in \phc$ were constructed within the context of Theorem \ref{Teo1}, featuring exactly three measures of maximal entropy, all of them are hyperbolic. Two of these measures have support in different transversally hyperbolic $su$-tori.  When the unstable foliation is minimal, it is proved that either $f$ has a unique measure of maximum entropy (with a zero central Lyapunov exponent), or $f$ has exactly two hyperbolic measures with central Lyapunov exponents of opposite signs; however, the $su$-torus do not exist in this scenario.
\end{remark}



\end{document}